\documentclass[10pt,a4paper]{amsart}

\usepackage{graphicx}
\usepackage{amssymb}
\usepackage{verbatim}
\usepackage[T1]{fontenc}
 \usepackage{pifont}
\usepackage[usenames, dvipsnames]{color}
\usepackage{amsmath,amsthm, mathtools}%

\usepackage{changes}
\usepackage{cancel}
\setdeletedmarkup{\cancel{#1}}

\newtheorem{theorem}{Theorem}[section] %sub
\newtheorem{lemma}[theorem]{Lemma}
\newtheorem{proposition}[theorem]{Proposition}
\newtheorem{corollary}[theorem]{Corollary}

\theoremstyle{definition}
\newtheorem{definition}[theorem]{Definition}
\newtheorem{remark}[theorem]{Remark}
\newtheorem{example}[theorem]{Example}

\newtheorem{question}[theorem]{Question}

\newcommand{\popo}{\mathbb{P}^1 \times \mathbb{P}^1}

\newcommand{\N}{\mathbb{N}}

\newcommand{\con}{\,{\big\|} \,}
\newcommand{\card}[1]{\left|#1\right|}

\begin{document}

\title{A numerical property of Hilbert functions and lex segment ideals
}
\author{Giuseppe Favacchio}
\address{Dipartimento di Matematica e Informatica\\ Viale A. Doria, 6 - 95100 - Catania, Italy}
\email{favacchio@dmi.unict.it} \urladdr{www.dmi.unict.it/~gfavacchio}

\thanks{Version:\today }

\maketitle

\begin{abstract} We introduce the \textit{fractal expansions}, sequences of integers associated to a number. We show that these sequences characterize the $O$-sequences and encode some informations on lex segment ideals. 
Moreover, we introduce a numerical functions called \textit{fractal functions}, and we show that fractal functions can be used to classify the Hilbert functions of bigraded algebras.
	
\end{abstract}

%%%%%%%%%%%%%%%%%%%%%%%%%%%%%%%%%%%%%%%
%%%%%%%%%%%%%%%%%%%%%%%%%%%%%%%%%%%%%%%

%\renewcommand{\thetheorem}{\thesection.\arabic{theorem}}

\section{Introduction}
In Commutative Algebra (and other fields of pure mathematics) it often happens that easy numerical conditions describe some deeper algebraic results. A significant example are the $O$-sequences.  

Let $S:=k[x_1,\ldots, x_n]$ be the standard graded polynomial ring and let $I\subseteq S$ be a homogeneous ideal. The quotient ring  $S/I$ is called a \textit{standard graded k-algebra}. The Hilbert function of  $S/I$ is defined as $H_{S/I}: \mathbb N\to \mathbb N $ such that
\[H_{S/I}(t):=\dim_k\left(S/I\right)_t= \dim_k S_t - \dim_k I_t. \]

A famous theorem, due to Macaulay (cf. \cite{M}) and pointed out by Stanley (cf. \cite{St}), characterizes the numerical functions that are Hilbert functions of a standard graded $k$-algebra, i.e. the functions $H$ such that $H=H_{S/I}$ for some homogeneous ideal $I\subseteq S$.
To introduce this fundamental result we need some preparatory material.

Let $h,i>0$ be integers, we can uniquely write $h$ as
\begin{equation}\label{def:bin exp}
	h ={m_i\choose i} + {m_{i-1}\choose i-1}+\cdots+ {m_j\choose j} 
\end{equation}

where $m_i > m_{i-1} > \cdots > m_j \ge j \ge 1.$
This expression is called the $i$-\textit{binomial expansion} of the integer $h$.   

If $h>0$  has $i$-binomial expansion as in  \ref{def:bin exp}, then we set
\[h^{\langle i\rangle } = {m_i+1\choose i+1} + {m_i\choose i}+\cdots+ {m_j+1\choose j+1}.
\]
We use the convention that $0^{\langle i\rangle }=0.$

For example, since $7= {4 \choose 3} + {3\choose 2} $, the $3$-binomial expansion of $7$ is 
	and $7^{\langle 3\rangle }= {5 \choose 4} + {4\choose 3} = 9.$

\begin{definition}
	A sequence of non-negative integers $(h_0, h_1, h_2 \cdots )$ is called an $O$-sequence if
	\begin{itemize}
		\item[i)] $h_0 = 1$;
		\item[ii)] $h_{i+1} \le h_i^{\langle i\rangle }$
		for all $i>0.$
	\end{itemize}
	An $O$-sequence is said to have maximal growth from degree $i$
	to degree $i + 1$ if $h_{i+1} = h_i^{\langle i\rangle }$
\end{definition} 
We are now ready to enunciate the Macaulay's Theorem. It characterizes the Hilbert function of standard graded $k$-algebras bounding the growth from any degree to the next. The proof of this theorem and more details about $O$-sequences are also discussed in \cite{BH}  Chapter 4.
We represent $H_{S/I}$ as a sequence of integers $(h_0, h_1, h_2, \cdots)$ where $h_t:=H_{S/I}(t)$.
\begin{theorem}[Macaulay, \cite{M}]\label{Macaulay}
	Let $H:=(h_0, h_1, h_2, \cdots)$ be a sequence of integers, then the following are equivalent:
	\begin{itemize}
		\item[1)] $H$ is the Hilbert function of a standard graded $k$-algebra;
		\item[2)] $H$ is an $O$-sequence.
	\end{itemize}
\end{theorem}

It is therefore interesting to find an extension of the above theorem to the multi-graded case.
Multi-graded Hilbert functions arise in many contexts and  properties related to the Hilbert function of multi-graded algebras are currently studied. See for instance \cite{GVT} and \cite{PS} for several examples.
The generalization of Macaulay's theorem to multi-graded rings is an open problem. A first answer was given by the author in \cite{F} where the Hilbert functions of a bigraded algebra in $k[x_1,x_2,y_1,y_2]$ are classified.  

In this paper, see Section \ref{s.Bilex Algebras} Theorem \ref{main}, we generalize the Macaulay's theorem to any bigraded algebra. First, in Section \ref{s.fractal}, we introduce $\Phi(n)$ a list of finite sequences called fractal expansion of $n$. Then we define a \textit{coherent truncation} of these vectors and we show that these objects are strictly related to the $O$-sequences.  Indeed, in Section \ref{s.Fractal and Homol} we show that they also characterize the Hilbert function of standard graded $k$-algebras. Furthermore, we show that these sequences can be used to compute the Betti numbers of a lex ideal.  

The computer program CoCoA \cite{C} was indispensable for all the computations. 

\section{The expansion of a Fractal sequence} \label{s.fractal}
In this section we introduce a sequence of tuples, called \textit{coherent fractal growth}, and we study its properties. In the main result of this section, Theorem \ref{t.bound on growth}, we prove that these sequences have the same behavior of the $O$-sequences.

Roughly speaking a numerical sequence $\sigma$ is called fractal if once we delete the first occurrence of each number it remains identical to the original. Such property thus implies we can repeat this process indefinitely, and $\sigma$ contains infinitely many copy of itself. It has something like a \textit{fractal} behavior. See \cite{K} for a formal definition and further properties.

For instance, one can show that the sequence
	$$\sigma=( 1,1,2,1,2,3,1,2,3,4,1,2,3,4,5,\ldots)$$ is fractal. Indeed, after removing the first occurrence of each number we get a sequence that is the same as the starting one
	$$(\cancel{1},1,\cancel{2},1,2,\cancel{3},1,2,3,\cancel{4},1,2,3,4,\cancel{5},\ldots)= \sigma.$$

%The sequence of positive integers will be called $\N_+$-sequences.

We introduce some notation. Given a positive integer $a\in \N$ we denote by $[a]:=(1,2,\ldots, a)\in \N^a$ the tuple of length $a$ consisting af all the positive integers less then or equal to $a$ written in increasing order. 

Given a (finite or infinite) sequence $\sigma$ of positive integers we construct a new sequence, named the expansion of $\sigma,$ denoted by $[\sigma]$. If $\sigma:=(s_1,s_2,s_3,\ldots)$ we set $[\sigma]:=[s_1]\con [s_2],\con [s_3]\con \cdots,$ where the symbol  "$\con$" denotes the associative operation of concatenation of two vectors. E.g. $(3,5,4)\con(2,3)=(3,5,4,2,3).$   

This construction can be recursively applied. We denote by  $[\sigma]^d:=[[\sigma]^{d-1}]$ where we set $[\sigma]^0:=\sigma$. 
For a positive integer $a$, we also denote by $[a]^d:=[[a]^{d-1}]$, where $[a]^0:=(a).$

For instance we have $$[3]=(1,2,3),\ \ [3]^2=(1,1,2,1,2,3), \ \ [3]^3=(1,1,1,2,1,1,2,1,2,3).$$

\begin{lemma}\label{l.frac^d}
	Let $\sigma:=(s_1,s_2,s_3,\ldots)$ be a sequence of positive integers. Then $[\sigma]^d=[s_1]^d\con [s_2]^d \con [s_3]^d \con \cdots.$
\end{lemma} 
 \begin{proof} If $d=0$ the statement is true. Assume $d>0.$
 	By definition $[\sigma]^d=[[\sigma]^{d-1}]$ then by the inductive hypothesis we have
 	$$[\sigma]^d=[ [s_1]^{d-1}]\con [[s_2]^{d-1}] \con [[s_3]^{d-1}] \con \cdots =[s_1]^d\con [s_2]^d \con [s_3]^d \con \cdots.$$
 \end{proof}

\begin{corollary}\label{c.[a]^d=[1]^d-1...}
	Let $a\in \N$ be a positive integer. Then $$[a]^d=[1]^{d-1}\con[2]^{d-1}\con\cdots\con [a]^{d-1}.$$
\end{corollary} 
\begin{proof}
 The statement follows by Lemma \ref{l.frac^d} since  $[a]^d=[(1,2,\ldots, a)]^{d-1}.$
\end{proof}

\begin{remark} The sequence  
	$[\N]:=[1]\con [2]\con \cdots \con  [n] \con \cdots $
	is a fractal sequence. 
\end{remark}

Given a sequence $\sigma$,  the symbols $\sigma_i$ and $\sigma(i)$ both denote the $i$-th entry of $\sigma$. 
If $\sigma$ is finite then $|\sigma|$ denotes the number of entries and $\sum\sigma$ their sum. We use the convention that these values are "$\infty$" for infinite sequences of positive integers.   
For instance, $[3]^3_4=2$, $\card{[3]^3}=10$ and $\sum[3]^3=15.$

Throughout this paper we use the convention that, for a finite sequence  $\sigma$,  the notation $\sigma_a$ or $\sigma(a)$  implies $a\le |\sigma|.$  

\begin{remark}\label{r.card and sum}Note that, for a finite sequence of positive integers $\sigma$, the definition of $[\sigma]$ easily implies the equality $\sum\sigma=\card{[\sigma]}.$
\end{remark}

Given a positive integer $n$, we define the fractal expansion of $n$ as the list
$$\Phi(n):=([n]^0,[n]^1,\ldots,[n]^d,\ldots).$$
Each element in $\Phi(n)$ is a finite sequence of positive integers. In the following lemma we compute the number of their entries.

\begin{lemma}\label{l.DecCardandSum}	
	Let $n$ be a positive integer. Then $\left|[n]^d\right|={ d+n-1 \choose d}$ and $\sum[n]^d={ d+n \choose d+1}.$	
\end{lemma}
\begin{proof} 
	By definition we have $\card{[n]^0}={ n-1 \choose 0}=1$ and $\card{[n]^1}={ n \choose 1}=n.$  
	By Corollary \ref{c.[a]^d=[1]^d-1...} we have $[n]^d=[1]^{d-1}\con [2]^{d-1}\con\cdots\con [n]^{d-1}$, therefore 
	$$\card{[n]^d}=\sum_{j=1}^{n}\card{[j]^{d-1}}= \sum_{j=1}^{n}{ d+j-2 \choose n-1}= { d+n-1 \choose d}.$$ 
	Moreover, by Remark \ref{r.card and sum} we have $\sum[n]^d= \card{[n]^{d+1}}={ d+n \choose d+1}.$
\end{proof}

Next lemma introduces a way to decompose a number as sum of binomial coefficients that is slight different to the Macaulay decomposition.  
We use the convention that ${a \choose b}=0$ whenever $a<b.$
\begin{lemma}\label{l.fractal dec}Let $d$ be a positive integer. Any $a\in \N$ can be written uniquely in the form
	\begin{equation}\label{eq.decomposition}
	a= {k_d \choose d} + {k_{d-1} \choose d-1} +\cdots +{k_2 \choose 2} + {k_1 \choose 1}
	\end{equation}
	where $k_d>k_{d-1} > \cdots > k_2\ge k_1 \ge 1 $
\end{lemma}
\begin{proof}
	In order to prove the existence, we choose $k_d$ maximal such that ${k_d \choose d}<a.$ If $d=1$ the statement is trivial. 
	If $d=2$ then let $k_2$ be the maximal integer such that $a- {k_2 \choose 2}>0.$ Then set $k_1:=a- {k_2 \choose 2}.$ Since $a\le {k_2+1 \choose 2} $ we have $k_1=a-{k_2 \choose 2}\le {k_2+1 \choose 2}-{k_2 \choose 2}=k_2$.
    Assume $d>2$, and let $k_d$ be the maximum integer such that  $a-{k_d \choose d}>0$. By induction, $a-{k_d \choose d}$ can be written as  $a-{k_d \choose d}=\sum_{i=1}^{d-1} {k_i \choose i},$ where $k_{d-1}>\cdots >k_2\ge k_1\ge 1.$ 
	Since $k_1\ge 1$ we have   $a-{k_d \choose d}> {k_{d-1} \choose d-1}.$
	Moreover, since ${k_d+1 \choose d}\ge a,$ it follows that
	$$ {k_d \choose d-1} = {k_{d}+1 \choose d}- {k_{d} \choose d} \ge a - {k_{d} \choose d} > {k_{d-1} \choose d-1}.$$
	Hence $k_d>k_{d-1}.$
	The uniqueness follows by induction on $d$.  If $d=1$ it is trivial. Now let $d>1$ and assume for each $b\in \N$ the uniqueness of the decomposition  \begin{equation}
	b= {k'_{d-1} \choose d-1} +\cdots +{k'_2 \choose 2} + {k'_1 \choose 1}
	\end{equation}
	where $k'_{d-1} > \cdots > k'_2\ge k'_1 \ge 1.$
	Let $a\in \N$ and let \begin{equation}\label{eq.a}
	a= {k_d \choose d} + {k_{d-1} \choose d-1} +\cdots +{k_2 \choose 2} + {k_1 \choose 1}
	\end{equation} be a decomposition of $a.$ 
	Then we claim that $k_d$ is the maximal integer such that ${k_d \choose d}<a.$
	Indeed, if ${k_d+1 \choose d}<a$ we get 
	$a> {k_d+1 \choose d}\ge {k_d \choose d} + {k_{d-1} \choose d-1} +\cdots +{k_2 \choose 2} + {k_1 \choose 1}=a.$
	Consider now the number $a - {k_{d} \choose d}$. From equation \ref{eq.a} we have  $a - {k_{d} \choose d}=\sum_{i=1}^{d-1} {k_i \choose i} $ where $k_{d-1}>\cdots >k_2\ge k_1\ge 1.$ The uniqueness of this decomposition follows from inductive hypothesis. This also implies the uniqueness of equation \ref{eq.a}.
\end{proof}

\begin{remark}
The decomposition in Lemma \ref{l.fractal dec} is different from the Macaulay decomposition since it is always required that $k_1\ge 1.$	Moreover, for any $j\ge 2$ we only have that  $k_j\ge j-1$, thus some binomial coefficient could be zero. For instance, we have $1= {d-1 \choose d} + {d-2 \choose d-1} +\cdots +{1 \choose 2} + {1 \choose 1}$ where the first $d-1$ binomial coefficients in the sum are equal to zero. 
\end{remark}

\begin{definition}\label{d.fractal decomposition}
We refer to equation \ref{eq.decomposition} as the $d$-fractal decomposition of $a.$ We denote by 
$$[a]^{(d)}:=(k_d-d+2,k_{d-1}-d+3,\ldots,k_4-2, k_3-1, k_2, k_1 )\in \N^d.$$
We call these numbers the $d$-fractal coefficients of $a.$
\end{definition}

$[a]^{(d)}$ is a not increasing sequence of positive integers.	
Indeed, $[a]^{(d)}_d>0$ and by construction $k_1\le k_2.$ Moreover, for $j>1$, since $k_j>k_{j-1}$, we have $k_j-j +2\ge k_{j-1}-j+1$.

Next result explains the name "$d$-fractal decomposition". We show that $k_1$ is the $a$-th entry in  $[n]^d$ i.e. in $[\N]^{d-1}$.
We use the convention that $[0]^d$ is the empty sequence and $\card{[0]^d}=0.$  
\begin{theorem}\label{t.k_d} $[\N]^{d-1}_a=[a]^{(d)}_d.$ 
\end{theorem}
\begin{proof}Assume $n$ is large enough, we will show that $[n]^{d}_a=[a]^{(d)}_d.$ If $d=1$ then $[n]^1_a= a$ and $a={ a\choose 1}$ thus $[a]^{(1)}_1=a=[n]^1_a.$
	We now assume $d>1$. Let	$a= {k_d \choose d} + {k_{d-1} \choose d-1} +\cdots +{k_2 \choose 2} + {k_1 \choose 1}$ 
	be the $d$-fractal decomposition of $a$.
	Note that, since ${k_{d-1} \choose d-1} +\cdots +{k_2 \choose 2} + {k_1 \choose 1}$ is the $(d-1)$ fractal decomposition of $a- {k_d \choose d}$, by the inductive hypothesis, we have $k_1=[n]^{d-1}_{a- {k_d \choose d}}.$ 
	Therefore, we only have to show that $[n]^{d}_{a}=[n]^{d-1}_{a- {k_d \choose d}}$.
	Since $a\le {k_d+1 \choose d}$, by Lemma \ref{l.DecCardandSum}, we have $\card{[k_d-d+2]^{d}}\ge a> \card{[k_d-d+1]^{d}}$. 
	Since from Corollary \ref{c.[a]^d=[1]^d-1...} we have the equivalence $[n]^d= [1]^{d-1}\con[2]^{d-1}\con\cdots\con[k_d-d+1]^{d-1}\con[k_d-d+2]^{d-1}\con\cdots$,  it follows that the $a$-th entry in $[n]^d$ is the $\left(a- {k_d \choose d}\right)$-th in $[k_d-d+2]^{d-1},$ i.e.,
	$[n]^d_a=[k_d-d+2]^{d-1}_{a-{k_d \choose d}}=[n]^{d-1}_{a- {k_d \choose d}}.$
\end{proof}

We introduce the lexicographic order for elements in $\N^d$. 
Given $\alpha,\beta\in \N^d$ then $\alpha<_{lex}\beta$ if and only if for some $i\le d$ we have $\alpha_j=\beta_j$ for $j<i$ and  $\alpha_i<\beta_i$.  
The following lemma is crucial for our intent. We prove that the $d$-fractal coefficients have a good behavior with respect the lex order.

\begin{lemma}\label{l.lexord}
	$[a]^{(d)}<_{lex} [b]^{(d)}$ iff $a<b$.
\end{lemma}
\begin{proof}If $d=1$ the assertion is trivial.
	Let $d>1$ and let $a,b$ be two integers with fractal decomposition  $$a= {a_d \choose d} + {a_{d-1} \choose d-1} +\cdots +{a_2 \choose 2} + {a_1 \choose 1},\ \ a_d> \cdots> a_2\ge a_1\ge 1$$
	and 
	$$b= {b_d \choose d} + {b_{d-1} \choose d-1} +\cdots +{b_2 \choose 2} + {b_1 \choose 1},\ \ b_d> \cdots> b_2\ge b_1\ge 1.$$
	If $[a]^{(d)}<_{lex} [b]^{(d)}$ then there is an index $j$ such that $[a]^{(d)}_i= [b]^{(d)}_i$ for any $i<j$ and $[a]^{(d)}_j< [b]^{(d)}_j.$ 
	Hence, ${a_i \choose i} = {b_i \choose i}$  for any $i<j$ and ${a_j \choose j} < {b_j \choose j}.$ If $j=d$ then easily $b>a$, otherwise we have 
	$$b> {b_d \choose d} + {b_{d-1} \choose d-1} +\cdots +{b_j \choose j}\ge {a_d \choose d} + {a_{d-1} \choose d-1} +\cdots +{a_i +1 \choose j} \ge a.$$
	
	Vice versa let $b>a$. We claim that $b_d\ge a_d$. Indeed, if  $b_d<a_d$ we get
	 $ a> {a_d \choose d}\ge{b_d+1 \choose d}\ge b$ contradicting  $b>a.$ 
	So  if $b_d> a_d$ we are done otherwise the statement follows by induction. 
\end{proof}

Given two sequences $\tau$ and $\sigma$ we say that $\tau$ is a truncation of $\sigma$ if $|\tau|\le|\sigma|$ and $\tau(j)=\sigma(j)$ for any $j\le|\tau|.$ For instance $(1,1,2,1,2)$ is a truncation of $[3]^2.$

Next definition introduces the main tool of the paper. The \textit{coherent fractal growths} are suitable truncations of the elements in the fractal expansion of $n.$ 
\begin{definition}
We say that  $T:=(\tau_0, \tau_1, \tau_2, \ldots)$ is a  coherent fractal growth  if  $\tau_0:=(n)$ and
$\tau_j$ is a truncation of $[\tau_{j-1}]$ for each $j\ge 1.$
\end{definition}

For instance $((3), (1,2,3), (1,1,2,1,2), (1,1,1,2,1,1), (1,1,1,1))$ is a coherent fractal growth. Indeed one can check that each elements is truncation of the expansion of the previous one. On the other hand, for instance, $((3), (1,2,3), (1,1,2), (1,1,1,2,1)) $ is not a coherent fractal growth. Indeed, $(1,1,1,2,1)$ is not a truncation of $[(1,1,2)]=(1,1,1,2)$.

\begin{remark}\label{r.coherent growth}
	Note that, in a coherent fractal growth, $\tau_d$ consists of the first $\card{\tau_d}$ elements in $[n]^d.$ Moreover, the length of the elements in  $T:=(\tau_0, \tau_1, \tau_2, \ldots )$, a coherent fractal growth, is bounded for any $d$.  Indeed, by Remark \ref{r.card and sum}  we have \begin{equation}\label{eq.coherent growth}
	|\tau_d|\le \sum \tau_{d-1}
	\end{equation} for each $d\ge 1$.
\end{remark}

In the last part of this section we prove that the bound in Remark \ref{r.coherent growth} is equivalent to the binomial expansion for a $O$-sequence. In order to relate the coherent fractal growth with $O$-sequences we need the following lemma. 
%It uses the equality ${a\choose b}={a-1\choose b}+{a-1\choose b-1}$.
\begin{lemma}\label{l.growth}Let $[a]^{(d)}=(c_d,c_{d-1},\ldots, c_1)$ be the $d$-fractal coefficients of $a$.
	Then  the $(d+1)$-fractal coefficients of $a^{\langle d \rangle}$ are $$[a^{\langle d \rangle}]^{(d+1)}=(c_d,c_{d-1},\ldots,c_2, c_1,c_1).$$
\end{lemma}
\begin{proof}The $d$-fractal decomposition of $a$ is, by definition,
	$$a={c_d+d-2\choose d}+{c_{d-1}+d-3\choose d-1}+\cdots+ {c_3+1\choose 3}+ {c_2\choose 2}+{c_1\choose 1}.$$
	If $c_1<c_2$, we get the Macaulay decomposition of $a$ by removing the binomials ${j \choose i}$ equal to 0. Since ${j \choose i}=0$ implies ${j+1 \choose i+1}=0$ we have $$a^{\langle d\rangle}={c_d+d-1\choose d+1}+{c_{d}+d-2\choose d}+\cdots+ {c_3+2\choose 4}+ {c_2+1\choose 3}+{c_1+1\choose 2}.$$
	Since ${c_1+1\choose 2}={c_1\choose 2}+{c_1\choose 1}$ we are done.
	Now we consider the case $c_1=c_2$. Then we get the following decomposition of $a$ $$a={c_d+d-2\choose d}+{c_{d-1}+d-3\choose d-1}+\cdots+ {c_3+1\choose 3}+ {c_2+1\choose 2}.$$
	Thus, if $c_3>c_2$, this representation is the Macaulay decomposition of $a$ once we remove the binomials ${j \choose i}$ equal to 0.
	$$a^{\langle d\rangle}={c_d+d-1\choose d+1}+{c_{d-1}+d-2\choose d}+\cdots+ {c_3+2\choose 4}+ {c_2+2\choose 3} =$$
	$$={c_d+d-1\choose d+1}+{c_{d-1}+d-2\choose d}+\cdots+ {c_3+2\choose 4}+ {c_2+1\choose 3}+ {c_1+1\choose 2} =$$
	$$={c_d+d-1\choose d+1}+{c_{d-1}+d-2\choose d}+\cdots+ {c_3+2\choose 4}+ {c_2+1\choose 3}+ {c_1\choose 2}+ {c_1\choose 1}.$$
	The proof follows in a finite number of steps by iterating this argument. 	
 \end{proof}

The following theorem is the main result of this section. We show that the length of the elements in a coherent fractal growth is an $O$-sequence. 
\begin{theorem}\label{t.bound on growth} 
	Let $T=(\tau_0, \tau_1, \tau_2 \ldots )$ be a list of truncations of $\Phi(n)=((n),[n],[n]^2,\ldots)$. Then the following are equivalent
	\begin{itemize}
		\item[i)] $T=(\tau_0, \tau_1, \tau_2 \ldots)$ is a coherent fractal growth;
		\item[ii)] $(|\tau_0|, |\tau_1|, |\tau_2|,  \ldots)$ is an O-sequence. %$|\tau_{d+1}|\le |\tau_{d}|^{\langle d \rangle}$, for each $d\ge 0$.
	\end{itemize}
\end{theorem}
\begin{proof} 
	In order to prove $(i)\Rightarrow(ii)$ we need to show $|\tau_{d+1}|\le |\tau_{d}|^{\langle d \rangle}$, for each $d\ge 0$. 
	Set $a=|\tau_{d}|$ and take the $d$-fractal decomposition of $a$
	$$a={k_d\choose d}+{k_{d-1}\choose d-1}+\cdots+ {k_3\choose 3}+ {k_2\choose 2}+{k_1\choose 1}.$$
	
	If $a={n+d-1 \choose d}=\card{[n]^d}$ the statement follows by Lemma \ref{l.DecCardandSum}. Assume now $a<{n+d-1 \choose d}$.
	Since $\tau_d$ is a truncation of $[n]^d=[1]^{d-1}\con [2]^{d-1}\con \cdots \con[n]^{d-1} $ 
	and, by Lemma \ref{l.DecCardandSum},  $$\card{[k_d-d+1]^d} = {k_d \choose d}<a\le{k_d+1 \choose d}=\card{[k_d-d+2]^d} $$
	we get (denoted by $[0]$ the empty sequence)
	$$\tau_d=[k_d-d+1]^{d} \con \tau_d'$$
	where $\tau_d'$ is a truncation of $[k_d-d+2]^{d-1}$.
	Therefore, iterating this argument, we have
	$$\tau_d=[k_d-d+1]^{d} \con [k_{d-1}-d+2]^{d-1} \con \cdots \con [k_3-1]^3 \con [k_2-1]^2 \con [k_1].$$
	By equation \ref{eq.coherent growth} in Remark \ref{r.coherent growth}, we have $|\tau_{d+1}|\le \sum \tau_{d}= \sum_{i=1}^d {k_i+1 \choose d+1}=$ $$={k_d+1\choose d+1}+{k_{d-1}+1\choose d}+\cdots+  {k_2+1\choose 3}+{k_1\choose 2}+{k_1\choose 1}.$$
	This sum, by Lemma \ref{l.growth}, is equal to $a^{ \langle d \rangle}$.
	
	Vice versa, to prove $(ii)\Rightarrow (i),$ we have to show that, for each $d\ge 0,$ the sequence $\tau_{d+1}$ is a truncation of $[\tau_d]$ i.e. $|\tau_{d+1}|\le\card{[\tau_{d}]}$. 
	By using the same argument as above, we get $|\tau_{d}|^{\langle d \rangle} = \sum \tau_{d}$. Then, the statement follows since$|\tau_{d+1}|\le |\tau_{d}|^{\langle d \rangle},$ by hypothesis, and $\card{[\tau_{d}]}= \sum \tau_{d},$  by Remark \ref{r.coherent growth}.

\end{proof}

Let's check, for instance, that $H:=(1,3,3,4)$ is an $O$-sequence. We write a sequence of truncations of $\Phi(3)$ of length $1,3,3,4$ respectively. We get
 $$T:=((3), (1,2,3), (1,1,2), (1,1,1,2)).$$
 It is a coherent fractal growth. Indeed, by definition each sequence is a truncation of the bracket of the previous one, e.g. $[(1,2,3)]=(1,1,2)\con(1,2,3)$ and $[(1,1,2)]=(1,1,1,2).$ %Hence $(1,3,3,4)$ is an $O$-sequence.
Now, we check that $H:=(1,3,5,8)$ is not an $O$-sequence. Indeed, take a coherent fractal growth   
$$((3), (1,2,3), (1,1,2,1,2), \tau_3)$$
where the first three sequences are  truncations  of length $1,3,5$ of the elements in  $\Phi(3)$.
Then $\tau_3$ should be a truncation of $[(1,1,2,1,2)]=(1,1,1,2,1,1,2)$
that has length 7. Thus $\card{\tau_3}\le 7$ and $(1,3,5,8)$ is not an $O$-sequence. 

\section{Fractal expansions and Homological Invariants}\label{s.Fractal and Homol}
In Section \ref{s.fractal} we introduced a novel approach to describe the $O$-sequences. In this section we show the "algebraic"  meaning of a coherent fractal growth. We directly relate these sequences to lex segment ideals and their homological invariants. In particular, the Eliahou-Kervaire formula is naturally applied to our case. Therefore, the fractal expansion of $n$ is used in Proposition \ref{p.EK2} to compute the Betti numbers of a lex algebra. See Section 2.1.2 in \cite{HH} for a complete discussion on monomial orders, including the lexicographic order.

Let $a$ and $d$ be positive integers. Let $[a]^{(d)}:=(c_d, c_{d-1}, \ldots, c_2, c_1)$ be the $d$-fractal coefficient of $a$, see Lemma \ref{l.fractal dec} and Definition \ref{d.fractal decomposition}.
We associate to $a$ and $d$ a monomial $X_a^{(d)}$ of degree $d$ in the variables $X:=\{x_1, \ldots x_n\}$ in such a way  $$X_a^{(d)}:=x_{c_1}x_{c_2}\cdots x_{c_{d-1}} x_{c_d}.$$ 

Vice versa, a monomial $T=x_{c_1}x_{c_2}\cdots x_{c_{d-1}} x_{c_d}$ of degree $d$ in the variables in $X$ identifies a $d$-upla $(c_d,c_{d-1}, \ldots, c_1)$ such that $c_i\ge c_{i+1}$ for each $i$.

\begin{remark}\label{r.i-th monomial in lex order}
	An immediate consequence of Lemma \ref{l.lexord} is that 	$X_a^{(d)}>_{lex} X_b^{(d)}$ if and only if $a>b$,  with respect the lexicographic order induced by $x_1<x_2<\cdots<x_n.$
	Therefore, with respect the same order, $X_a^{(d)}$ is the $a$-th greatest monomial (in the variables in $X$) of degree $d$.
\end{remark}
Let $S:=k[X]=k[x_1,\ldots, x_n]$ be the standard graded polynomial ring. 

	We set $$\mathcal G(X)^{(d)}_{\le t}:=\{X_a^{(d)} \ |\ a\le t \}$$
	and 
	$$\mathcal G(X)^{(d)}_{> t}:=\{X_a^{(d)} \ |\ a> t \}.$$
	
	$S_d$ is spanned by the monomials in $\mathcal G(X)^{(d)}_{\le t} \cup \mathcal G(X)^{(d)}_{> t}.$
	
By Remark \ref{r.i-th monomial in lex order}	$G(X)^{(d)}_{> t}$ is a lex set of monomials of degree $d$  with respect the degree lexicographic order $x_1<x_2<\cdots<x_n.$

	Given $T:=\{\tau_0, \tau_1, \tau_2,\ldots \}$ a coherent fractal growth, we set 
$I(T)_d:= \langle \mathcal G(X)^{(d)}_{> |\tau_d|}\rangle_k$
the $k$-vector space spanned by the monomial in $\mathcal G(X)^{(d)}_{> |\tau_d|}$.
Then, by Theorem, \ref{t.bound on growth} and Theorem \ref{Macaulay}, the following result holds.
\begin{proposition}
$I(T):=\oplus_d I_d(T)\subseteq R$ is a lex segment ideal and $H_{R/I}(d)=|\tau_d|.$
\end{proposition}

%We associate to $\tau_d$ a set of bilex monomial of degree $d.$ $A_d:= \{X_a^{(d)} \ |\ a\le |\tau_d| \}$ 

Given a minimal free resolution of a lex segment ideal
$$0 \to \bigoplus_{j} S(-j)^{\beta_{p,p+j}} \to\cdots\to \bigoplus_{j} S(-j)^{\beta_{1,1+j}} \to S\to S/I\to 0  $$
the Betti numbers can be computed by the Eliahou-Kervaire  formula, see \cite{EK}. See also \cite{HH} equation 7.7 of Section 7.

\begin{theorem}[Eliahou-Kervaire formula]\label{t.EK}
	Let $I$ be a lex segment ideal. For $u \in\mathcal G(I)$,  a monomial minimal generator of $I$, let $m(u)$ denotes the largest index $j$ such that $x_j$ divides $u.$ Let $m_{kj}$ be the number of monomials  $u\in \mathcal G(I)$  with $m(u) = k$.
	  Then
	$$\beta_{i,i+j}(S/I)=\sum_{u\in \mathcal G(I)_j} {m(u)-1\choose i} =\sum_{i=1}^n {k-1\choose i-1}m_{kj}.$$
\end{theorem}
This result can be written in terms of coherent fractal growth.
\begin{proposition}\label{p.EK2}
	Given $T:=\{\tau_0, \tau_1, \tau_2,\ldots \}$ a coherent fractal growth. Then 
	$$\beta_{i,i+j}(S/I(T))=\sum_{a=|\tau_j|+1 }^{\card{[\tau_{i-1}]}} {[a]^{(j)}_j-1\choose i}=\sum_{i=1}^n {k-1\choose i-1}w_{kj} $$
	where $w_{kj}$ is the number of occurrence of "$k$" in $\tau_i'$ with $[\tau_{i-1}]=\tau_i\con \tau_i'.$
\end{proposition}
\begin{proof}
	It is an immediate consequence of Theorem \ref{t.EK} and Theorem \ref{t.k_d}.
\end{proof}
%\textbf{(controllare i gradi dei Betti)}
\section{The Hilbert function of a bigraded algebra and the Fractal Functions}\label{s.Bilex Algebras}
Let $k$ be an infinite field, and let $R :=k[x_1, \cdots x_n , y_1,\cdots, y_m ]$ be the polynomial ring in $n+m$ indeterminates with the grading defined by $\deg x_i = (1, 0)$ and $\deg y_j = (0, 1).$ Then $R=\oplus_{(i,j)\in \N^2}R_{(i,j)}$ where $R_{(i,j)}$ denotes the set of all homogeneous elements in $R$ of degree $(i,j).$ Moreover,  $R_{(i,j)}$ is generated, as a $k$-vector space, by the monomials $x_1^{i_1}\cdots x_n^{i_n}y_1^{j_1}\cdots y_m^{j_m}$ such that $i_1+\cdots+i_n=i$ and $j_1+\cdots+j_m=j.$
An ideal $I\subseteq R$ is called a $bigraded$ ideal if it is generated by homogeneous elements with respect to this grading. A bigraded algebra $R/I$ is the quotient of $R$ with a bigraded ideal $I.$  The Hilbert function of a bigraded algebra $R/I$ is defined such that $H_{R/I}:\N^2\to \N$ and 
$H_{R/I}(i,j):=\dim_k (R/I)_{(i,j)}=\dim_k R_{(i,j)} -\dim_k I_{(i,j)}$ where $I_{(i,j)} =I\cap R_{(i,j)}$ is the set of the bihomogeneous elements of degree $(i,j)$ in $I.$

From now on, we will work with the degree lexicographical order on $R$ induced by $x_n > \cdots > x_1 > y_m >\cdots > y_1.$ With this ordering, we recall the definition of bilex ideal, introduced and studied in \cite{ACD}. We refer to \cite{ACD} for all preliminaries and for further results on bilex ideals.

\begin{definition}[\cite{ACD}, Definition 4.4]\label{Def4.4} 
	%Let us denote by $x$ a monomial in $k[x_1,x_2]$ and by $y$ a monomials in $k[y_1,y_2].$
	A set of monomials $L\subseteq R_{(i,j)}$ is called $bilex$ if for every monomial $uv\in L,$ where $u\in R_{(i,0)}$ and $v\in R_{(0,j)},$  the following conditions are satisfied:
	\begin{itemize}
		\item if $u'\in R_{(i,0)}$ and $u' > u$, then $u'v \in L;$
		\item if $v'\in R_{(0,j)}$ and $v' > v$, then $uv' \in L.$
	\end{itemize}
	A monomial ideal $I \subseteq R$ is called a $bilex\ ideal$ if $I_{(i,j)}$ is generated as $k$-vector space by a bilex set of monomials, for every $i,j\ge 0$.%:=I\cap R_{(i,j)}
\end{definition}

Bilex ideals play a crucial role in the study of the Hilbert function of bigraded algebras.
\begin{theorem}[\cite{ACD},Theorem 4.14] Let $J\subseteq R$ be a bigraded ideal. Then there exists a bilex ideal $I$
	such that $H_{R/I}=H_{R/J}.$% have the same Hilbert function.
\end{theorem}

In \cite{F} was solved the question of characterize the Hilbert functions of bigraded algebras of $k[x_1,x_2,y_1,y_2]$ by introducing the Ferrers functions. In this section we generalize these functions by introducing the fractal functions, see Definition \ref{d.F}. We prove, Theorem \ref{main}, that these  classify the Hilbert functions of bigraded algebras.
%We introduce a numerical function in order to characterize the HF of  bigraded algebras.

We need some preparatory material.  We denote by $\mathcal M^{a\times b}(U)$ the set of all the matrices with size $(a,b)$ -  $a$ rows and $b$ columns -  and entries in a set $U\subseteq \N.$   
Given a matrix $M=(m_{ij})\in \mathcal M^{a\times b}(U)$ %with $a$ rows and $b$ columns and entries $m_{ij}\in U\subseteq k$
we denote by $$\sum M:=\sum_{i\le a}\sum_{j\le b} m_{i,j}.$$
%named the weight of $M.$

Next definition introduces the objects we need in this section. The $\le$ symbol denote the partial order on $\N^2$ given by componentwise comparison. 
\begin{definition}\label{d.FerrersMatrix}
	A \textit{Ferrers matrix} of size $(a,b)$ is a matrix $M=(m_{ij})\in \mathcal M(\{0,1\})^{a\times b}$ such that \begin{center}
		if $m_{ij}=1$ then $m_{i'j'}=1$ for any $(i',j')\le (i,j)$. 
	\end{center}
	We set by $\mathcal F^{a\times b}$ the family of all the Ferrers matrices of size $(a,b).$
\end{definition}

In the following definition we introduce expansions of a matrix. 
\begin{definition} Let $M\in \mathcal M(U)^{a\times b}$ be a matrix of size $(a,b)$ and let $\textbf{v}:=(v_{1},\ldots,v_{a} )\in \N^a$ and $\textbf{w}:=(w_{1},\ldots,w_{b} )\in \N^b$ be vectors of non negative integers. We denote by $M^{\langle \textbf{v}, \bullet \rangle}$ an element in $M(U)^{\sum\textbf{v}\times b}$ constructed by \begin{center}
		repeating the $i$-th row of $M$ $v_{i}$ times	, for $i=1,\ldots,a$.
	\end{center} 
	We denote by $M^{\langle \bullet, \textbf{w} \rangle}$ an element in $M(U)^{a\times \sum\textbf{v}_2}$ constructed by \begin{center}
		repeating the $j$-th column of $M$ $w_{j}$ times, for $j=1,\ldots,b$.
	\end{center} 
\end{definition}

\begin{remark}
	The expansions of a Ferrers matrix  are also Ferres matrices. Take, for instance, $$M=\left(\begin{array}{ccc}
	1 & 1 & 1 \\
	1 & 1 & 0 \\
	1 & 1 & 0 \\
	1 & 0 & 0 \\
	\end{array}\right)\in \mathcal F^{(4, 3)}.$$
	Set $\textbf{v}:=(2,1,0,3)$ and $\textbf{w}:= (3,1,3).$ Then
{\small	$$M^{\langle \textbf{v},\bullet \rangle}=\left(\begin{array}{ccc}
	1 & 1 & 1 \\
	1 & 1 & 1 \\
	1 & 1 & 0 \\
	%1 & 1 & 0 \\
	1 & 0 & 0 \\
	1 & 0 & 0 \\
	1 & 0 & 0 \\
	\end{array}\right)\in \mathcal F^{(6, 3)}\ \ \ \ M^{\langle \bullet, \textbf{w} \rangle}=\left(\begin{array}{ccccccc}
	1 & 1 & 1 &  1 & 1  & 1  & 1 \\
	1 & 1 & 1 &  1 & 0  & 0  & 0  \\
	1 & 1 & 1 &  1 & 0  & 0  & 0 \\
	1 & 1 & 1 &  0 & 0  & 0  & 0 \\
	\end{array}\right)\in \mathcal F^{(4, 7)}.$$}
\end{remark}	

%We define the fractal expansion of a matrix.

Given $M, N\in \mathcal F^{a\times b}$ we say that $M\le N$ if and only if $m_{ij}\le n_{ij}$ for any $i,j.$

We are ready to introduce the fractal functions.  
\begin{definition}\label{d.F} Let $H:\N^2\to \N$ be a numerical function. %Set $V:=Vec(n)=\{\vv_0,\vv_1,\vv_2,\ldots\}$ and $W:=Vec(m)=\{\vw_0,\vw_1,\vw_2,\ldots\}$ 
	We say that $H$ is a \textit{fractal function} if $H(0,0)=1$ and, for any $(i,j)\in \N^2,$ there exists a matrix of $M_{ij}\in \mathcal F^{{i-1+n \choose n-1}\times {j-1+m \choose m-1}}$ with $\sum M_{ij}=H(i,j)$ and such that all the matrices satisfy the condition  
	$$\left\{\begin{array}{ll}
	M_{ij}\le M_{i-1,j}^{\langle[n]^{i-1},\bullet\rangle}& \text{if}\ i>0 \\
	& \\
	M_{ij}\le M_{i,j-1}^{\langle\bullet,[m]^{j-1}\rangle}& \text{if}\ j>0 \\
	\end{array}\right.$$
\end{definition}

\begin{remark}
	Let $H:\mathbb N^2 \to \mathbb N$ be the numerical function
	$$H(i,j)= {i-1+n\choose n-1}{j-1+m\choose m-1}.$$ For any $i,j\in \mathbb N,$ there is only one element in $M_{ij}\in\mathcal F^{{i-1+n\choose n-1}\times{j-1+m\choose m-1}}$ satisfying the condition in Definition \ref{d.F} that is the matrix with all "1" entries. Therefore $H$ is a fractal function. 
\end{remark}

\begin{remark}If $n=m=2$, the definition of fractal functions agrees with Definition 3.3. in \cite{F}.  Indeed it is enough to write each partition $\alpha_{ij}=(a_1,a_2, \ldots)$ as a matrix $M_{ij}=(m_{hk})\in\mathcal F^{(i+1)\times (j+1)}$ where $m_{hk}=1$ iff $k\le a_k$ otherwise $m_{hk}=0$. In this case the expansions are given by the elements in $\Phi(2):=((2),(1,2),(1,1,2),(1,1,1,2),\ldots)$.
\end{remark}

In general the matrices $M_{ij}$ are not uniquely determined by the conditions of Definition \ref{d.F}, even in "small" cases. 
\begin{example}\label{e.is fractal?} Set $n=m=2$ and let $H:\N^2\to \N$ be the numerical function
		$$ H:= \begin{array}{l|lllll}
		& 0 & 1 & 2 & 3 & \ldots\\
		\hline
		0  & 1 & 2 & 3 & 0 &\ldots\\
		1  & 2 & 4 & 3 & 0 & \ldots\\
		2  & 3 & 3 & 3 & 0 &\ldots\\
		3  & 0 & 0 & 0 & 0 &\ldots\\
		\end{array}
		$$
The only possibility for $M_{11}$ is  $$M_{11}=\left(\begin{array}{ccc}
1 & 1 \\
1 & 1 \\
\end{array}\right)\in \mathcal F^{(2, 2)}.$$
Thus, we have no restriction on $M_{21}$ and $M_{12}$ but the number of non zero entries. 
$$M_{12}\in \left\{\left(\begin{array}{ccc}
1 & 1& 1 \\
0 & 0& 0 \\
\end{array}\right),\  \left(\begin{array}{ccc}
1 & 1& 0 \\
1 & 0& 0 \\
\end{array}\right) \right\} \subseteq \mathcal F^{(2, 3)}.$$
$$M_{21}\in \left\{\left(\begin{array}{ccc}
1 & 0 \\
1 & 0 \\
1 & 0 \\
\end{array}\right),\   \left(\begin{array}{ccc}
1 & 1 \\
1 & 0 \\
0 & 0 \\
\end{array}\right) \right\} \subseteq \mathcal F^{(3, 2)}.$$

Now we check if all these choices are allowed. We have to look at the conditions on $M_{22}$.
We have $M_{22}\le M_{1,2}^{\langle(1,2),\bullet\rangle}$ and  $M_{22}\le M_{2,1}^{\langle\bullet, (1,2)\rangle}$. In the next table we collect all the possibilities for $M_{22}$ depending on $M_{1,2}$ and $M_{2,1}.$ 
\[%\arraycolsep=1pt\def\arraystretch{1}
\begin{array}{c||c|c}
M_{22}\le & \left(\begin{array}{ccc}
	1 & 1& 1 \\
	0 & 0& 0 \\
	\end{array}\right)& \left(\begin{array}{ccc}
	1 & 1& 0 \\
	1 & 0& 0 \\
	\end{array}\right)\\[0,3cm]
\hline
\hline
 \left(\begin{array}{ccc}
1 & 0 \\
1 & 0 \\
1 & 0 \\
\end{array}\right) & 
\left(\begin{array}{ccc}
1 & 0 & 0 \\
0 & 0 & 0 \\
0 & 0 & 0 \\
\end{array}\right)
&
\left(\begin{array}{ccc}
1 & 0 & 0 \\
1 & 0 & 0 \\
1 & 0 & 0 \\
\end{array}\right)
\\[0,5cm]
\hline
\left(\begin{array}{ccc}
1 & 1 \\
1 & 0 \\
0 & 0 \\
\end{array}\right)	& 
\left(\begin{array}{ccc}
1 & 1 & 1 \\
0 & 0 & 0 \\
0 & 0 & 0 \\
\end{array}\right)
&
\left(\begin{array}{ccc}
1 & 1 & 0 \\
1 & 0 & 0 \\
0 & 0 & 0 \\
\end{array}\right)
\\
\end{array}
\]
Since $H(2,2)=3$, we note that \begin{itemize}
	\item[--] only one of the above four matrices in the table must be excluded;
	\item[--] in the other three cases $M_{22}$ is uniquely determined by $M_{12}$ and $M_{21}$.
\end{itemize}
  Therefore, $H$ is a fractal function and three different set of matrices satisfy the conditions in Definition \ref{d.F}. 	
\end{example}

%\section{The Hilbert function of a bigraded algebra}\label{sec:main}
In the following we denote by $X:=\{x_1, \cdots x_n\}$ and $Y:=\{y_1,\cdots, y_m\}$ the set of the variables of degree $(1,0)$ and $(0,1)$ respectively. 
%We denote for short $$R :=k[X,Y]= k[x_1, \cdots x_n , y_1,\cdots, y_m ]$$
%$$R_X :=k[X]= k[x_1, \cdots x_n]$$
%$$R_Y :=k[Y]= k[y_1,\cdots, y_m ]$$

Next lemma is useful for our purpose. 	It is an immediate consequence of Lemma \ref{l.growth}.
\begin{lemma}\label{l.x_1 cdot X^d}
	$x_1\cdot X^{(d)}_a=X^{(d+1)}_{a^{\langle d \rangle}}.$  
\end{lemma}

To shorten the notation we set $\alpha_i:={n+i-1 \choose n-1 }$ and $\beta_j:={m+j-1 \choose m-1 }$.
In order to relate fractal functions and Hilbert functions of bigraded algebras we need to introduce a correspondence between Ferrers matrices and monomials.

Let $M=(m_{ab})\in \mathcal F^{\alpha_i\times\beta_j}.$ We denote by $\lambda(M)_{(ij)}$ the set of the monomials 
$$\lambda(M):=\{ X^{(i)}_a\cdot Y^{(j)}_b\ |\ m_{ab}=0 \}.$$   

Let $L\subseteq R_{(i,j)}$ be a bilex set of monomials of bidegree $(i,j)$. We denote by $\mu(L)\in \mathcal M(\{0,1\})^{\alpha_i\times \beta_j}$ the matrix $(m_{ab})$ such that $m_{ab}=1$ iff $X^{(i)}_a\cdot Y^{(j)}_b\notin L$ otherwise $m_{ab}=0$.

\begin{proposition}\label{p.lambda is a Lex, mu is a Ferrers}There is a one to one correspondence between bilex sets of monomials of degree $(i,j)$ and elements in $\mathcal F^{{\alpha_i}\times{\beta_j}}.$
\end{proposition}
\begin{proof}
 Let $M=(m_{ab})\in \mathcal F^{\alpha_i\times\beta_j}$, we claim that $$\lambda(M)\ \text{ is a bilex set of monomials of bidegree}\ (i,j).$$	
 To prove the claim we use Lemma \ref{l.lexord} and Remark \ref{r.i-th monomial in lex order}. 
 Let $X_a^{(i)}\cdot Y_{b}^{(j)}$ be an element in $\lambda(M)$ and $X_u^{(i)}> X_a^{(i)}.$ Since $(u,b)>(a,b)$ and  $m_{ab}=0$ we get $m_{ub}=0$ i.e.   $X_u^{(i)}\cdot Y_{b}^{(j)}\in \lambda(M).$ 
 In a similar way it follows that $X_a^{(i)}\cdot Y_{v}^{(j)}\in \lambda(M)$ for $v>b.$ 	
	
 Let $L\subseteq R_{(i,j)}$ be a bilex set of monomials of bidegree $(i,j)$. We claim that $$\mu(L)\in \mathcal F^{\alpha_i\times \beta_j}.$$
 The claim follows by using Lemma \ref{l.lexord} and Remark \ref{r.i-th monomial in lex order}. Indeed, say $m_{ab}=0$ for an entry of $\mu(L)$. This implies   $X^{(i)}_a\cdot Y^{(j)}_b\in L$. Thus, for $u>a$, we have $X^{(i)}_u\cdot Y^{(j)}_b\in L$ i.e. $m_{ub}=0$. Analogously, we see that $m_{av}=0$ for $v>b.$

\end{proof}

We are ready to prove the main result of this paper. 
\begin{theorem}\label{main}
	Let $H:\N\times\N\to \N$ be a numerical function. Then the following are equivalent
	\begin{enumerate}
		\item[1)] $H$ is a fractal function;
		\item[2)] $H=H_{R/I}$ for some bilex ideal $I\subseteq R= k[x_1,\ldots ,x_n,y_1,\ldots, y_m ].$
	\end{enumerate}  
\end{theorem}
\begin{proof}
		$(1)\Rightarrow (2)$ Let $H$ be a fractal function. 
		For each $(i,j)\in \N^2$, let $I_{(i,j)}$ be the $k$-vector space spanned by the elements in $\lambda(M_{ij})$. We claim that $I:=\oplus_{(i,j)\in\N^2} I_{(i,j)}$ is an ideal of $R.$  To prove the claim, it is enough to show that if $X_a^{(i)}\cdot Y_b^{(j)}\in I_{(i,j)}$ then $x_u\cdot X_a^{(i)}\cdot Y_b^{(j)}\in I_{(i+1,j)}$ for any $x_u\in X$ and $y_v\cdot X_a^{(i)}\cdot Y_b^{(j)}\in I_{(i,j+1)}$ for any $y_v\in Y.$ We have, see Lemma \ref{l.x_1 cdot X^d},
		$$x_u\cdot X_a^{(i)}\cdot Y_b^{(j)} \ge x_1\cdot X_a^{(i)}\cdot Y_b^{(j)} =X_{a^{\langle i \rangle}}^{(i+1)}\cdot Y_b^{(j)}.$$
		Then, by Definition \ref{d.F} and Theorem \ref{t.bound on growth}, the entry $(a^{\langle i \rangle},b)$ in the matrix $M_{i+1,j}$ is $0$. Thus $X_{a^{\langle i \rangle}}^{(i+1)}\cdot Y_b^{(j)}\in I_{(i+1,j)}$ and furthermore $x_u\cdot X_a^{(i)}\cdot Y_b^{(j)}\in I_{(i+1,j)}.$
	In a similar way it follows that $y_v\cdot X_a^{(i)}\cdot Y_b^{(j)}\in I_{(i,j+1)}$.
	
		$(2)\Rightarrow (1)$   Let $I\subseteq R$ be a bilex ideal such that $H_{R/I}=H.$ Set $M_{ij}:=\mu(I_{ij})$, we claim that the $M_{ij}$s satisfy the condition in Definition \ref{d.F}. By Theorem \ref{t.bound on growth} it is enough to show that if $M_{ij}(a,b)=0$ (the entry $(a,b)$ in $M_{ij}$ is $0$), then also $M_{i+1,j}(a^{\langle i \rangle },b)=0$ (the entry $(a^{\langle i \rangle },b)$ in $M_{i+1j}$ is $0$). Set $J:=(X_u^{(h)}\ |\ M_{hj}(u,b)=0)$, then the claim is an immediate consequence of the fact that $J$ is a lex ideal of $k[X]$. 
\end{proof}

\begin{remark}
	Note that since the $M_{ij}$ are not uniquely determined, as shown in Example \ref{e.is fractal?}, then there could be several bilex ideals having the same Hilbert function.  
\end{remark}

%$$>>-----------------------------------------<<$$
The following question is motivated by the arguments in Section \ref{s.Fractal and Homol}.
\begin{question}
Can the bigraded Betti numbers of a bilex ideal $I=\oplus I_{(i,j)}$ be computed from the matrices $\mu(I_{(i,j)})$?
\end{question}

\end{document}